\documentclass[oneside,english]{amsart}
\usepackage[T1]{fontenc}
\usepackage[latin9]{inputenc}
\usepackage{geometry}
\usepackage{amsthm}
\usepackage{setspace}
\usepackage{esint}
\makeatletter
\numberwithin{equation}{section}
\numberwithin{figure}{section}
\theoremstyle{plain}
\newtheorem{thm}{\protect\theoremname}[section]
  \theoremstyle{definition}
  \newtheorem{defn}[thm]{\protect\definitionname}
  \theoremstyle{plain}
  \newtheorem{prop}[thm]{\protect\propositionname}
  \theoremstyle{plain}
  \newtheorem{lem}[thm]{\protect\lemmaname}
  \theoremstyle{plain}
  \newtheorem{cor}[thm]{\protect\corollaryname}
  \theoremstyle{remark}
  \newtheorem{rem}[thm]{\protect\remarkname}
  \theoremstyle{definition}
  \newtheorem{example}[thm]{\protect\examplename}

\makeatother

\usepackage{babel}
  \providecommand{\corollaryname}{Corollary}
  \providecommand{\definitionname}{Definition}
  \providecommand{\examplename}{Example}
  \providecommand{\lemmaname}{Lemma}
  \providecommand{\propositionname}{Proposition}
  \providecommand{\remarkname}{Remark}
\providecommand{\theoremname}{Theorem}

\begin{document}

\title{Planar Infinite-Horizon Optimal Control Problems with Weighted Average
Cost and Averaged Constraints,\\ Applied to Cheeger Sets.}

\author{Ido Bright}

\address{Department of Applied Mathematics, University of Washington, Seattle, WA 98195 USA; ibright@uw.edu}
\begin{abstract}
We establish a Poincaré-Bendixson type result for a weighted averaged infinite horizon problem in the plane, with and without averaged constraints. For the unconstrained problem, we establish the existence of a periodic optimal solution, and for constrained problem, we establish the existence of an optimal solution that alternates cyclicly between a finite number of periodic curves, depending on the number of constraints. Applications of these results are presented to the shape optimization problems of the Cheeger set and the generalized Cheeger set, and also to a singular limit of the one-dimensional Cahn-Hilliard equation.
\end{abstract}



\maketitle

\section{Introduction}

\global\long\def\seq#1{#1_{1},#1_{2},\dots}

\global\long\def\seqf#1{#1_{1}\left(\cdot\right),#1_{2}\left(\cdot\right),\dots}

\global\long\def\seqfn#1#2{#1_{1}\left(\cdot\right),#1_{2}\left(\cdot\right),\dots,#1_{#2}\left(\cdot\right)}

\global\long\def\pd#1#2{#1_{1}\left(\cdot\right),#1_{2}\left(\cdot\right),\dots,#1_{#2}\left(\cdot\right)}

\global\long\def\pd{\left(\cdot\right)}

\global\long\def\sM{\mathcal{M}}

\global\long\def\bRb{\mathbb{R}}

\global\long\def\semiline{\left[0,\infty\right)}

\global\long\def\rinf#1{#1\rightarrow\infty}

\global\long\def\linf#1{\lim_{#1\rightarrow\infty}}

\global\long\def\liinf#1{\liminf_{#1\rightarrow\infty}}

\global\long\def\lsinf#1{\limsup_{#1\rightarrow\infty}}

The celebrated Poincare-Bendixson was originally stated for smooth
planar ordinary differential equations, but since has been extended
to flows, semi-flows and differential inclusions. (See, Ciesielski
\cite{ciesielski2002poincar,PoincaretotheXXIstcentury} for a review
on the Poincaré-Bendixson theorem and its generalizations.). In the
context of infinite horizon optimization, Poincaré-Bendixson type
results were obtained for second order Lagrangian by Leizarowitz and
Mizel \cite{LeizMizel}, for control systems by Artstein and Bright
\cite{artbright}, and for discounted infinite-horizon problems by
Colonius and Sieveking \cite{colonius1989asymptotic} and in Bright
\cite{Bright201284}.

In this paper we extend the infinite horizon result to a wider class
of values that include weighted average optimization, with an averaged
constraint. Two motivations for studying this problem are: The planar
shape optimization problem of the Cheeger set and its generalization;
and the van der Waals-Cahn-Hilliard theory of phase transition \cite{van1893verhandel,cahn1958free}. 

A planar Cheeger set in a domain $\Omega\subset R^{2}$ is a set $E\subset\Omega$,
that maximizes the ratio between its area and the length of its boundary.
One over the maximal ratio is the Cheeger constant of the problem.
The Cheeger constant is used to bound the first eigenvalue of the
Laplacian (see, Cheeger \cite{cheeger1970lower}), and generalized
Cheeger sets appear in the study of landslides (see, Ionescu and Lachand-Robert
\cite{ionescu2005generalized}). Average constraints arrises in the
study of phase transition, where the steady states of a partial differential
equations is the minimizer of an energy functional with a conserved
quantity (in this case the conservation of mass), and when a singular
limit is present, qualitative properties can be studied through the
solution of an infinite horizon problem with an averaged constraint.

The structure of this paper is as follows. The following section contains
definitions used throughout the paper, the standing assumptions and
some lemmas and previous results used in the proof of the main result.
In Section \ref{sec:MainResults} the main results are presented.
In Section \ref{sec:Applications}, we present applications of the
main result, establishing the reduction of the planar Cheeger problem,
and its generalization, to infinite horizon optimization, and the
singular limit of the Cahn-Hilliard equation to a constrained optimization.
The last section verifies the main result.

\section{\label{sec:Notations-&-Lemmas}Notations \& Assumptions \& Lemmas}

\subsection{General Notations and Assumptions}

The following notations are used throughout the paper. We denote the
set of Reals by $R$, vectors in the Euclidean space $R^{2}$ by $x=\left(x^{1},x^{2}\right)$
and $y=\left(y^{1},y^{2}\right)$, and the standard Euclidean norm
by $\left|\cdot\right|$. Given a metric space $X$, we denote its
probability space by $P\left(X\right)$. The action of a measure $\mu\in P\left(X\right)$
on a continuos function $g\left(\cdot\right)\in C\left(X\right)$
is denote by $\mu\left(g\right)=\int_{X}g\left(x\right)\mu\left(dx\right)$,
and given a set of measures $\mathcal{S}\subset P\left(X\right)$
and a function $g\left(\cdot\right)=\left(g^{1}\left(\cdot\right),\dots,g^{m}\left(\cdot\right)\right)\in C\left(X,R^{m}\right)$
we denote the set $\left\{ \mu\left(g\right)|\mu\in\mathcal{N}\right\} \subset R^{m}$
as its \textit{realization} by $g\left(\cdot\right)$.

Throughout the paper we assume that the control system
\begin{equation}
\frac{dx}{dt}=f\left(x,u\right),\label{eq:ODE}
\end{equation}
satisfies the following conditions: The function $f\left(x,t\right)$
is continuos, and it satisfies Lipschitz conditions in $x$. The constraint
set $K\subset R^{2}$ is compact, and so is the control set $U\subset R^{d}$
. We consider the optimization only with respect to solutions, with
measurable controls satisfying $u\left(t\right)\in U$, for almost
every $t\ge0$, defined on $\semiline$ and satisfying $x\left(t\right)\in K$
for every $t\ge0$. A solution satisfying the latter conditions is
denoted as \textit{feasible}. We assume that there exists at least
one feasible curve.

To obtain the convexity of the set of limiting measures (see Definition
\ref{def:Limiting measures}), we require the following controllability
assumption, that holds, for example, when for every $x_{0}\in K$
the convex hull of $\left\{ f\left(x_{0},u\right)|u\in U\right\} \subset R^{2}$
contains an open ball around the origin.
\begin{defn}
\label{def:unif_cont}The control system (\ref{eq:ODE}) is \textit{uniformly
controllable} in $K\subset R^{2}$, if there exists a $T_{K}>0$,
such that it is possible to steer between any two points $x_{1},x_{2}\in K$
in time less than $T_{K}$, while staying in the constraint set $K$.
\end{defn}

\subsection{Relaxed Controls}

A relaxed controls, say $\nu\left(\cdot\right)$, is a Young measure
(see, Young \cite{young1980lectures}), where for each time point
$\nu\left(t\right)\in P\left(U\right)$. With relaxed controls the
dynamics turn to $f\left(x,\nu\right)=\int_{U}f\left(x,u\right)\mathbf{\nu}\left(du\right)$
and the cost follows an equivalently representation.

Relaxed controls are introduced in this problem by limits of feasible
curves. Indeed, any sequence of solutions of (\ref{eq:ODE}) defined
in a common bounded interval, contains a subsequence converging, in
the sense of Young measures, to a solution of (\ref{eq:ODE}), perhaps,
with relaxed controls.

The main result establishes an optimal periodic or stationary solution
with relaxed controls. It is well known that in any bounded interval,
any solution of the relaxed system can be approximated using only
regular controls (see, Warga \cite[Chapter IV]{warga1972optimal}),
and in \cite[Section 7]{artbright} the approximation on an infinite
time domain is discussed.

\subsection{Occupational Measures\protect \\
}

We consider the probability space $P\left(K\times U\right)$ endowed
with the topology of weak convergence, where $\mu_{i}\rightarrow\mu$
weakly if $\lim_{i\rightarrow\infty}\mu_{i}\left(g\right)=\mu\left(g\right)$
for every $g\in C\left(K\times U\right)$. By our assumptions $K\times U$
is compact which implies that the set $P\left(K\times U\right)$ is
weakly compact, namely, every sequence contains a converging subsequence
converging  to a measure in $P\left(K\times U\right)$. (See, Billingsley
\cite[Chapter 1 ]{billingsley1968convergence})
\begin{defn}
We define the \textit{occupational measure} $\mu\in P\left(K\times U\right)$
corresponding to a curve $\left(x\left(\cdot\right),u\left(\cdot\right)\right)$
defined on $\left[0,T\right]$, by 
\[
\mu\left(A\right)=\frac{1}{T}\lambda\left(t\in\left[0,t\right]|\left(x\left(t\right),u\left(t\right)\right)\in A\right),
\]
for every Borel sets $A\subset K\times U$, where $\lambda$ is the
Lebesgue measure. Notice, that for every $g\in C\left(K\times U\right)$
\end{defn}
\[
\mu\left(g\right)=\int_{K\times U}g\left(x,u\right)\mu\left(dx,du\right)=\frac{1}{T}\int_{0}^{T}g\left(x\left(t\right),u\left(t\right)\right)dt.
\]

\begin{defn}
\label{def:Limiting measures}\label{def: Limitting_meas} We denote
$\mu$ as a \emph{limiting occupational measure} if there exists a
feasible solution $\left(x\left(\cdot\right),u\left(\cdot\right)\right)$
of (\ref{eq:ODE}), and a sequence $T_{i}\rightarrow\infty$ such
that the sequence of occupational measures corresponding to the restriction
of $\left(x\left(\cdot\right),u\left(\cdot\right)\right)$ to the
intervals $\left[0,T_{i}\right]$ converges weakly to $\mu$. The
set of all limiting occupational measures is denoted by $\mathcal{M}\subset P\left(K\times U\right)$.
\end{defn}
With an additional controllability assumption a simple diagonalization
argument implies that $\mathcal{M}$ satisfies the following property.
\begin{prop}
\label{prop:M_convex_compact}If system (\ref{eq:ODE}) is uniformly
controllable (see Definition \ref{def:unif_cont}) then $\mathcal{M}$
is convex and closed in the weak topology.
\end{prop}
We now state a result from \cite[Lemma 8.2]{artbright} on the set
of occupational measures.
\begin{thm}
\label{thm:occ_meas_realization}Suppose $g:K\times U\rightarrow R^{n}$
is continuous. Let
\[
\Delta=\left\{ \mu\left(g\right)|\mu\in\mathcal{M}\right\} ,
\]
be a realization of the set of limiting occupational measures. Every
extreme point of the (not necessarily convex) set $\Delta$ corresponds
to a stationary solution or a periodic solution with image being a
Jordan curve. Namely, given an extreme point $y\in\Delta$, there
exists a feasible pair $\left(x_{p}\left(\cdot\right),u_{p}\left(\cdot\right)\right)$
defined on $\left[0,T_{p}\right]$, possibly, with relaxed controls,
with corresponding measure $\mu_{p}$ such that $x\left(0\right)=x\left(T_{p}\right)$
and
\[
y=\mu_{p}\left(g\right)=\int_{K\times U}g\left(x,u\right)\mu_{p}\left(dx,du\right)=\frac{1}{T_{p}}\int_{0}^{T}g\left(x_{p}\left(t\right),u_{p}\left(t\right)\right)dt.
\]

\end{thm}

\subsection{Convexity Notations and Lemmas\protect \\
}

We use the following notation and lemmas.
\begin{defn}
Let $\Delta\subset R^{n}$. We denote $H:\Delta\rightarrow R$ a quasi-convex
function if for every $z_{1},z_{2}\in\Delta$ 
\[
\max\left\{ H\left(\left(1-\lambda\right)z_{1}+\lambda z_{2}\right)|0\le\lambda\le1\right\} =\max\left\{ H\left(z_{1}\right),H\left(z_{2}\right)\right\} .
\]
It is denoted quasi-concave if for every $z_{1},z_{2}\in\Delta$ 
\[
\min\left\{ H\left(\left(1-\lambda\right)z_{1}+\lambda z_{2}\right)|0\le\lambda\le1\right\} =\min\left\{ H\left(z_{1}\right),H\left(z_{2}\right)\right\} .
\]
\end{defn}
\begin{lem}
Suppose $\Delta\subset R^{n}$ is a convex set and $\Pi\subset R^{n}$
is an affine subspace of dimension $n-d$. If the set $\Pi\cap\Delta$
is non-empty then each of its extreme points can be represented as
 a convex combination of $d+1$ extreme points of $\Delta$.\end{lem}
\begin{proof}
By Caratheodory's theorem, every point in $\Delta$ is a convex combination
of $n+1$ of its extreme points. If $y\in\Pi\cap\Delta$ can be expressed
as a convex combination of $k>d+1$ distinct extreme points of $\Delta$
with non-zero coefficients, then, locally, $\Delta$ contains a convex
set of dimension $k-1>d$ containing $y$ in its relative interior.
The intersection of the latter set with $\Pi$ contains a line centered
at $y$, thus $y$ is not an extreme point of $\Pi\cap\Delta$.\end{proof}
\begin{cor}
\label{Cor:cnv_intesect_halfspace_lemma}Suppose $\Delta\subset R^{n}$
is a convex set and $\Pi_{1},\dots,\Pi_{m}\subset R^{n}$ are affine
subspaces of dimension $n-d$. If the set $\Delta\cap\Pi_{1}\cap\cdots\cap\Pi_{m}$
is non-empty then each of its extreme points can be represented as
a convex combination of $d*m+1$ extreme points of $\Delta$.
\end{cor}

\section{\label{sec:MainResults}Main Results}

Our main result establishes a Poincaré-Bendixson type results for
a wide class of averaged infinite-horizon optimization problems. The
main problem we consider is the optimization of the ratio of integrals.
We present our main results for this value, but they hold for a wider
class of optimization problems. (See Remark \ref{rem:gen_val} and
Theorem \ref{thm:GEN_THM}.). 

We wish to minimize (or maximize, see, Remark \ref{rem:min-max})
\begin{equation}
v^{*}=\limsup_{T\rightarrow\infty}\frac{\int_{0}^{T}p\left(x\left(t\right),u\left(t\right)\right)dt}{\int_{0}^{T}q\left(x\left(t\right),u\left(t\right)\right)dt}\label{eq:Profit-Marg}
\end{equation}

with respect to all feasible solutions of Equation (\ref{eq:ODE}).

Two type of constraints are considered:
\begin{enumerate}
\item Cumulative constraint: Given a single continuous function $C_{1}\left(x,u\right)$
, we consider only feasible solutions satisfying 
\begin{equation}
\int_{0}^{T}C_{1}\left(x\left(t\right),u\left(t\right)\right)dt\le0,\qquad\forall T\ge0.\label{eq:Constraint}
\end{equation}

\item Averaged constraints: Given $m\ge1$ continuous functions $C_{1}\left(x,u\right),\dots,C_{m}\left(x,u\right)$
we consider only feasible solutions which for every $k=1,\dots,m$
satisfy either\\
%\addtocounter{equation}{-1} 
\begin{subequations}
\begin{equation}
\limsup_{T\rightarrow\infty}\frac{1}{T}\int_{0}^{T}C_{k}\left(x\left(t\right),u\left(t\right)\right)dt\le0,\label{eq:m-constraints_ineq}
\end{equation}
or
\begin{equation}
\lim_{T\rightarrow\infty}\frac{1}{T}\int_{0}^{T}C_{k}\left(x\left(t\right),u\left(t\right)\right)dt=0.\label{eq:m-constraints_eq}
\end{equation}
\end{subequations} 
\end{enumerate}
For the unconstrained problem, we extend a previous result in \cite{artbright},
establishing the existence of a periodic optimal solution. For the
constrained optimization problem, given by $m$ constraints, we establish
the existence of a solution that is either stationary or periodic
solution, or alternates between, at most, $m+1$ such solutions, in
a cyclic manner.

We now state our main results.
\begin{thm}
\label{thm:main theorem}Suppose $p\left(x,u\right)$ and $q\left(x,u\right)$
are continuous, and $q\left(x,u\right)$ is positive. The minimization
of (\ref{eq:Profit-Marg}) restricted to feasible solutions of (\ref{eq:ODE})
is attained by a stationary solution, or by a periodic solution with
image being a Jordan curve. Perhaps, using relaxed controls.
\end{thm}

\begin{thm}
\label{thm:main constraints}Suppose $p\left(x,u\right),q\left(x,u\right),C_{1}\left(x,u\right),\dots,C_{m}\left(x,u\right)$
are continuous, and $q\left(x,u\right)$ is positive. Consider the
minimization of (\ref{eq:Profit-Marg}) restricted to feasible solutions
of (\ref{eq:ODE}), satisfying $m$ averaged constraints of the form
(\ref{eq:m-constraints_ineq}) or (\ref{eq:m-constraints_eq}). If
there exists at least one feasible solution satisfying all the constraints,
then the minimum is attained by a curve that alternates cyclicly between
$m+1$ solutions of (\ref{eq:ODE}), each of which is either a periodic
solution with image being a Jordan curve, or a stationary solution.
Perhaps, using relaxed controls.

Moreover, when $m=1$, a stronger constraint of the form (\ref{eq:Constraint})
can be attained in a similar manner.
\end{thm}
The periodic or stationary optimal solution may be achieved only through
relaxed controls, however it can be approximated with regular controls,
in the same manner considered in \cite{artbright}.
\begin{rem}
The condition that $q\left(x,t\right)$ is positive can be relaxed
by requiring that its integral over every periodic curve is uniformly
bounded from below by a positive number. 
\end{rem}

\begin{rem}
\label{rem:min-max}The problem is stated as the minimization of the
limit superior, however, since the minimum is attained, we can replace
the limit superior by the limit. Also, replacing $p\left(x,u\right)$
by $-p\left(x,u\right)$ we see that our result generalize to the
maximization problems as well.
\end{rem}

\begin{rem}
\label{rem:gen_val}The minimization problem is stated for a weighted
average cost, however, it holds (see, Theorem \ref{thm:GEN_THM})
for a wider class of problems, such as the minimization of
\[
\inf_{\left(x\left(\cdot\right),u\left(\cdot\right)\right)}\limsup_{T\rightarrow\infty}\frac{1}{T}\left(\int_{0}^{T}p_{1}\left(x\left(t\right),u\left(t\right)\right)dt+\left|\int_{0}^{T}p_{2}\left(x\left(t\right),u\left(t\right)\right)dt\right|\right).
\]

\end{rem}

\section{\label{sec:Applications}Motivations \& Application: Shape Optimization
\& Partial Differential Equations}

In this section, we present an application of our result to the planar
shape optimization problem of the Cheeger set, and present qualitative
properties of the singular limit of the Cahn-Hilliard equation can
be obtained based on the corresponding infinite horizon optimization
problem with an averaged constraint.

\subsection{Shape Optimization: Cheeger Sets and Generalized Cheeger Sets\protect \\
}

The planar isoperimetric problem is the first known shape optimization
problem, where one seeks, amongst all sets $E\subset R^{2}$ with
a given perimeter, a set with maximal area. This problem was reformulated
to sets that maximize the ratio between their area and their perimeter,
in a given domain bounded $\Omega\subset R^{2}$. We denote the maximal
ratio by
\begin{equation}
v_{C}^{*}=\max_{E\subset\bar{\Omega}}V\left(E\right)=\max_{E\subset\bar{\Omega}}\frac{{Area}\left(E\right)}{{Length}\left(\partial E\right)}.\label{eq:Cheeger problem}
\end{equation}
The Cheeger constant is then defined by $1/v_{C}^{*}$. It is well
known that the maximum might not be unique, and that it is attained
by a set with boundary being a Jordan curve. The set maximizing (\ref{eq:Cheeger problem})
is called a Cheeger set, which we denote by 
\begin{equation}
E^{*}=\arg\max_{E\subset\bar{\Omega}}\frac{{Area}\left(E\right)}{{Length}\left(\partial E\right)},\;\;\partial E^{*}\mbox{ is a rectifiable Jordan curve}.\label{eq:Cheeger set Def}
\end{equation}

The Cheeger constant appear in the context of partial differential
equation. The Cheeger inequality bounds from below the largest eigenvalue
of the laplacian equation with homogenous boundary conditions by $\frac{1}{4}\left(v_{C}^{*}\right)^{-2}$
\cite{cheeger1970lower}. Cheeger sets are difficult to compute, and
only recently Kawohl and Lachand-Robert \cite{kawohl2006characterization}
provide an analytic characterization for convex planar domains. 

Cheeger sets have been generalized to generalized Cheeger sets, which
arise in applications in landslide modeling \cite{ionescu2005generalized}.
These sets maximize
\begin{equation}
v_{GC}^{*}=\max_{E\subset\Omega}V_{P;Q}\left(E\right)=\max_{E\subset\Omega}\frac{\int_{E}P\left(x\right)\mathcal{L}^{n}\left(dx\right)}{\int_{\partial^{*}E}Q\left(x\right)\mathcal{H}^{n-1}\left(dx\right)},\label{eq:Gen Cheeger problem}
\end{equation}
where $P\left(\cdot\right)$ and $Q\left(\cdot\right)$ are continuous,
$\mathcal{L}^{n}$ denotes the $n$-dimensional Lebegues measure,
and $\mathcal{H}^{n-1}$ denotes the $n-1$-dimensional Hausdorff
measure. Similar to the Cheeger set problem, the averaging property
implies that the optimal value is attained by a set with boundary
being a rectifiable Jordan curve.

The reduction to infinite-horizon optimization allows the use of tools
from infinite horizon optimization to solve these shape optimization
problems. In particular, numerical computation of the infinite horizon
problem in Proposition \ref{prop:cheeger_red} was performed to estimate
the Cheeger set and Cheeger constant, where they can be computed analytically
in order that the numerical result can be evaluated, using semi-infinite
linear programing (see, Gaitsgory, Rossomakhine \cite{gaitsgory2006linear}
and Gaitsgory, Rossomakhine and Thatcher \cite{Gaitsnumdisc}). Two
examples were computed: The 6x4 rectangle and the ovoid domain $\left(x^{2}+y^{2}\right)^{2}<x^{3}$.
In both cases the Cheeger constant was computed with accuracy of 4
digits, and for the rectangle the computed solution was within $2\cdot10^{-3}$
from the boundary of the theoretic Cheeger set.

In order to formulate problems (\ref{eq:Cheeger problem}) and (\ref{eq:Gen Cheeger problem})
as an infinite horizon control problems, we first express them by
an integral over a parametrized curve. 

Let $x\left(\cdot\right)=\left(x^{1}\left(\cdot\right),x^{2}\left(\cdot\right)\right)$
be a parametrized Jordan curve defined on $\left[0,T\right]$, with
interior $E$ and positive orientation. The area of $E$ can be expressed
as
\begin{equation}
{Area}\left(E\right)=\int_{0}^{T}x^{1}\left(t\right)\frac{dx^{2}}{dt}\left(t\right)dt,\label{eq:area formula}
\end{equation}
and the length of $\partial E$ as
\[
{Length}\left(\partial E\right)=\int_{0}^{T}\left|\frac{dx}{dt}\left(t\right)\right|dt.
\]
The integrals in expression (\ref{eq:Gen Cheeger problem}) can be
expressed, when $P\left(\cdot\right)\in C\left(\Omega\right)$, by
\[
\int_{E}P\left(x\right)\mathcal{L}^{n}\left(dx\right)=\int_{0}^{T}P_{1}\left(x\left(t\right)\right)\frac{dx^{2}}{dt}\left(t\right)dt,
\]
where 
\begin{equation}
P_{1}\left(\left(x^{1},x^{2}\right)\right)=\int_{0}^{x^{1}}P\left(\zeta,x^{2}\right)d\zeta,\label{eq:P1_def-1}
\end{equation}
and
\[
\int_{\partial^{*}E}Q\left(x\right)\mathcal{H}^{n-1}\left(dx\right)=\int_{0}^{T}Q\left(x\left(t\right)\right)\left|\frac{dx}{dt}\left(t\right)\right|dt.
\]

Thus, we have that
\[
V\left(E\right)=\frac{{Area}\left(E\right)}{{Length}\left(\partial E\right)}=\frac{\int_{0}^{T}x^{1}\left(t\right)\frac{dx^{2}}{dt}\left(t\right)dt}{\int_{0}^{T}\left|\frac{dx}{dt}\left(t\right)\right|dt},
\]
and 
\[
V_{P;Q}\left(E\right)=\frac{\int_{0}^{T}P_{1}\left(x\left(t\right)\right)\frac{dx^{2}}{dt}\left(t\right)dt}{\int_{0}^{T}Q\left(x\left(t\right)\right)\left|\frac{dx}{dt}\left(t\right)\right|dt}.
\]

Furthermore, when we consider solutions of (\ref{eq:ODE}), where
$f\left(x,u\right)=u$ and $U=\partial B\left(\mathbf{0},1\right)=\left\{ y\in R^{2}|\left|y\right|=1\right\} \subset R^{2}$,
the latter expressions reduce to 
\begin{equation}
V\left(E\right)=\frac{\int_{0}^{T}x^{1}\left(t\right)u^{2}\left(t\right)dt}{\int_{0}^{T}\left|f\left(x,u\right)\right|dt}=\frac{1}{T}\int_{0}^{T}x^{1}\left(t\right)u^{2}\left(t\right)dt,\label{eq:Cheeger set int identity}
\end{equation}
and
\[
V_{P;Q}\left(E\right)=\frac{\int_{0}^{T}P_{1}\left(x\left(t\right)\right)u^{2}\left(t\right)dt}{\int_{0}^{T}Q\left(x\left(t\right)\right)dt},
\]
where $u\left(t\right)=\left(u^{1}\left(t\right),u^{2}\left(t\right)\right)$. 

With these identities established, we verify the equivalence between
the shape optimization problem and the infinite horizon problem.
\begin{prop}
\label{prop:cheeger_red}Suppose the control system (\ref{eq:ODE})
defined by the vector field $f\left(x,u\right)=u$ with constraint
set $K=\bar{\Omega}$ and control set $U=\partial B\left(\mathbf{0},1\right)=\left\{ y\in R^{2}|\left|y\right|=1\right\} \subset R^{2}$.
The maximization problem 
\begin{equation}
v^{*}=\sup_{\left(x\left(\cdot\right),u\left(\cdot\right)\right)}\limsup_{T\rightarrow\infty}\frac{1}{T}\int_{0}^{T}x^{1}\left(t\right)u^{2}\left(t\right)dt,\label{eq:infinite-horizon cheeger}
\end{equation}
with respect to feasible solutions of (\ref{eq:ODE}), is attained
by a periodic curve with arc-length parametrization, such that its
image is a Jordan curve bounding a set $E$, which satisfies $v^{*}=V\left(E\right)=v_{C}^{*}$.\end{prop}
\begin{proof}
Theorem \ref{thm:main theorem} implies that the value is attained
either by a Jordan curve or by a stationary solution. The value of
a stationary solution is $0$, on the other hand, any Jordan curve,
with positive orientation, has positive value. Thus, the solution
is attained by $\left(x^{*}\left(\cdot\right),u^{*}\left(\cdot\right)\right)$
corresponding to a Jordan curve bounding a set $E_{J}$. The period
of this solution satisfies $T^{*}\ge{Length}\left(\partial E_{J}\right)$
(the inequality can be strict when relaxed controls are used). Thus,
by (\ref{eq:area formula}), $v^{*}={Area}\left(E_{J}\right)/T^{*}\le{Area}\left(E_{J}\right)/{Length}\left(\partial E_{J}\right)=V\left(E_{J}\right)\le v_{C}^{*}$,
which implies that $v^{*}\le v_{C}^{*}$.

We now prove that $v\ge v_{C}^{*}$. Let $E^{*}$ be given by (\ref{eq:Cheeger set Def}).
Applying (\ref{eq:Cheeger set int identity}) to a parametrized curve
with arc-length parametrization of $\partial E^{*}$ with positive
orientation, we conclude that $v^{*}\ge V\left(E^{*}\right)=v_{C}^{*}$,
and $v^{*}=V\left(E^{*}\right)=v_{C}^{*}$, which completes the proof.
\end{proof}
A Similar result holds for generalized Cheeger sets.
\begin{prop}
Suppose $P\left(\cdot\right),Q\left(\cdot\right)\in C\left(K\right)$
are positive, and $P_{1}\left(\cdot\right)$ is defined by (\ref{eq:P1_def-1}).
Consider the maximization problem

\begin{equation}
v^{*}=\sup_{\left(x\left(\cdot\right),u\left(\cdot\right)\right)}\limsup_{T\rightarrow\infty}\frac{\int_{0}^{T}P_{1}\left(x\left(t\right)\right)\frac{dx^{2}}{dt}\left(t\right)dt}{\int_{0}^{T}Q\left(x\left(t\right)\right)dt},\label{eq:infinite-horizon cheeger-1}
\end{equation}
with respect to feasible solutions of (\ref{eq:ODE}), given by the
vector field $f\left(x,u\right)=u$ and the sets $K=\bar{\Omega}$
and $U=\partial B\left(\mathbf{0},1\right)=\left\{ y\in R^{2}|\left|y\right|=1\right\} \subset R^{2}$.
The maximum is attained by a periodic curve with arc-length parametrization,
such that its image is a Jordan curve bounding a set $E$, which satisfies
$V\left(E\right)=v^{*}=v_{GC}^{*}$.
\end{prop}
This result can be generalized to the case when $Q\left(x,u\right)$
depends on the control.
\begin{prop}
Suppose $P\left(\cdot\right)\in C\left(K\right)$ is positive, $P_{1}\left(x\right)$
is defined by (\ref{eq:P1_def-1}) and the integral \textup{over every
periodic} solution of $Q\left(x,u\right)\in C\left(K\times U\right)$
is bounded from below by some $\eta>0$. Consider the maximization
problem

\begin{equation}
v^{*}=\sup_{\left(x\left(\cdot\right),u\left(\cdot\right)\right)}\limsup_{T\rightarrow\infty}\frac{\int_{0}^{T}P_{1}\left(x\left(t\right)\right)\frac{dx^{2}}{dt}\left(t\right)dt}{\int_{0}^{T}Q\left(x\left(t\right),u\left(t\right)\right)dt},\label{eq:infinite-horizon cheeger-1-1}
\end{equation}
with respect to feasible solutions of (\ref{eq:ODE}) given by the
vector field $f\left(x,u\right)=u$ and the sets $K=\Omega$ and $U=\partial B\left(\mathbf{0},1\right)=\left\{ y\in R^{2}|\left|y\right|=1\right\} \subset R^{2}$.
The optimal solution is attained by a Jordan curve. Perhaps, with
relaxed controls.\end{prop}
\begin{proof}
By our assumption on $Q\left(x,u\right)$ the minimization of 
\[
v_{Q}^{*}=\limsup_{T\rightarrow\infty}\frac{1}{T}\int_{0}^{T}Q\left(x\left(t\right),u\left(t\right)\right)dt
\]
is attained by a periodic solution. This implies that $v_{Q}^{*}\ge\eta>0$
and we can apply Theorem \ref{thm:GEN_THM}, and the proof follows
the observation that the value of a stationary solution is 0. 
\end{proof}

\subsection{Singular Limits of Partial Differential Equations\protect \\
}

Constrained optimization appears in the study of singular limits partial
differential equations, with steady state solutions minimizing an
energy functional with constrained. As an example, we study 1-dimensional
phase transitions, modeled by the van der Waals-Cahn-Hilliard theory.
We show how our main result can be employed to obtain qualitative
properties of the singular limit.

The limit of a solution of the Cahn-Hilliard equation 
\[
\frac{d}{dt}u_{\epsilon}\left(x,t\right)=\frac{d^{2}}{dx^{2}}\left(-\epsilon^{2}\frac{d^{2}}{dx^{2}}u_{\epsilon}\left(x,t\right)+W\left(u_{\epsilon}\left(\zeta,\tau\right)\right)\right),\quad u_{\epsilon}\left(\cdot,0\right)=u_{0}\left(\cdot\right),u_{\epsilon}\left(0,\cdot\right)=-1,\ u\left(1,\cdot\right)=1,
\]
where $\epsilon$ is a small parameter, is also the minimizer of the
van der Waals free energy given by
\begin{equation}
E_{\epsilon}=\int_{0}^{1}\left(W\left(u_{\epsilon}\left(x\right)\right)+\frac{\epsilon^{2}}{2}\left(u_{\epsilon}'\left(x\right)\right)^{2}\right)dx,\ u_{\epsilon}\left(0\right)=-1,\ u_{\epsilon}\left(1\right)=1,\label{eq:vanderWaals}
\end{equation}
where $u_{\epsilon}'\left(x\right)=\frac{d}{dx}u_{\epsilon}\left(x\right)$,
constrained by the conservation of mass
\begin{equation}
\int_{0}^{1}u_{\epsilon}\left(x\right)dx=\int_{0}^{1}u_{0}\left(x\right)dx=M.\label{eq:vdw_cnst}
\end{equation}

Our interest in this problem is in the limit of $u_{\epsilon}\left(\cdot\right)$
as $\epsilon$ goes to zero, and in $v^{**}=\liminf_{\epsilon\rightarrow0}E_{\epsilon}$.
Assuming sufficient (quadratic) growth conditions on $W\left(\cdot\right)$
it is easy to see that for all $\epsilon$ small enough, the minimizers
are uniformly bounded and uniformly Lipschitz.

By \cite[Proposition 6.1]{artbright} we can reduce the singular limit
problem to the constrained infinite horizon problem minimizing
\begin{equation}
v^{*}=\limsup_{T\rightarrow\infty}\frac{1}{T}\int_{0}^{T}\left(W\left(u\left(\zeta\right)\right)+\frac{\left(u'\left(\zeta\right)\right)^{2}}{2}\right)d\zeta,\label{eq:infhorvdW}
\end{equation}
with free initial condition and average constraint 
\begin{equation}
\lim_{T\rightarrow\infty}\frac{1}{T}\int_{0}^{T}W\left(u\left(\zeta\right)\right)d\zeta=M,\label{eq:infvdWcnt}
\end{equation}
with the same constraint set and Lipschitz constant as in the singular
limit. The reduction implies that $v^{*}=v^{**}$ and that the occupational
measure of solutions $u_{\epsilon}$ converge to an occupational measure
of an optimal solution of the infinite horizon equation. Note, that
\cite[Proposition 6.1]{artbright} considers the unconstrained case,
however, it can easily be extended to the constrained case we are
considering.

Artstein and Leizarowitz \cite{artstein2002singularly} show that
the minimum of unconstrained first order scalar Lagrangians is attained
by stationary solution. Thus, according to the methods in Section
\ref{sec:Proof-of-Main} there is an optimal measure which is a convex
combination of two measures corresponding to stationary solutions,
which for the double well potential, studied in Carr, Gurtin and Slemrod
\cite{carr1984structured}, is in fact unique. Thus, by \cite[Proposition 6.1]{artbright},
the singular limit is concentrated at two points as expected by \cite{carr1984structured}.

The Lagrangian in (\ref{eq:vanderWaals}) is a first order scalar
Lagrangian, for second order Lagrangians, with a single averaged constraint,
Theorem \ref{thm:main constraints} suggests, the appearance of approximate
``piecewise periodic'' optimal solutions for small $\epsilon$.
Namely, solutions that first approximately follow one periodic solution,
with velocity of order $O\left(\epsilon^{-1}\right)$, for a time
period of order $O\left(1\right)$, and then continues to approximately
follow a second periodic solution till $t=1$.

\section{\label{sec:Proof-of-Main}Proof of Main Result}

The existence of an optimal solution to our problem relays on a convexity
property of the set of limiting occupational measures $\mathcal{M}$
(see, Definition \ref{def:Limiting measures}). To employ this property,
we first restate the optimization problem using occupational measures,
reducing (\ref{eq:Profit-Marg}) to 
\[
v^{*}=\inf_{\mu\in\mathcal{M}}\mu\left(p\right)/\mu\left(q\right).
\]

The main theorems follows from the general result below. 
\begin{thm}
\label{thm:GEN_THM}Let $g\left(x,u\right)=\left(g^{1}\left(x,u\right),\dots,g^{n}\left(x,u\right)\right)\in C\left(K\times U,R^{n}\right)$,
$\Delta=\left\{ \mu\left(g\right)\in R^{n}|\mu\in\mathcal{M}\right\} $
be the realization of $\mathcal{M}$ by $g\left(x,u\right)$, and
$V\in C\left(\Delta\right)$ be quasi-concave. 
\begin{enumerate}
\item \label{enu:unconstrnd}The optimization problem
\begin{equation}
v^{*}=\inf_{\mu\in\mathcal{M}}V\left(\mu\left(g^{1}\right),\dots,\mu\left(g^{n}\right)\right),\label{eq:gen_opt_crit}
\end{equation}
attains its optimal solution by a measure corresponding to a periodic
or stationary solution of (\ref{eq:ODE}).
\item \label{cnstrnd}If system (\ref{eq:ODE}) is uniformly controllable
(Definition \ref{def:unif_cont}) then the constrained optimization
problem
\[
v^{*}=\inf_{\mu\in\mathcal{M^{C}}}V_{}\left(\mu_{i}\left(g^{1}\right),\dots,\mu_{i}\left(g^{n}\right)\right),
\]
where 
\[
\mathcal{M}^{C}=\left\{ \mu\in\mathcal{M}|\mu\left(g^{k}\right)\le0\mbox{ for }k=1,\dots m\right\} \neq\emptyset,
\]
attains its minimum by a measure $\mu^{*}\in\mathcal{M}^{C}$ such
that

\begin{enumerate}
\item $V\left(\mu^{*}\left(g^{1}\right),\dots,\mu^{*}\left(g^{n}\right)\right)=v^{*}$.
\item \label{enu:(b)}The measure $\mu^{*}$ is a convex combination of
$m+1$ occupational measures, corresponding to stationary or periodic
solutions of (\ref{eq:ODE}).
\item There is a feasible solution, corresponding to the measure $\mu^{*}$
, that alternates between the $m+1$ stationary or periodic solutions
from (b).
\item \label{enu:one_constraint}When $m=1$ a solution $\left(x^{*}\left(\cdot\right),u^{*}\left(\cdot\right)\right)$
of the form (c) exists, such that\textup{
\[
\int_{0}^{T}g^{1}\left(x^{*}\left(t\right),u^{*}\left(t\right)\right)dt\le0
\]
for every $T>0$.}
\end{enumerate}
\end{enumerate}
\end{thm}
\begin{proof}
The function $V\left(\cdot\right)$ is quasi-concave, thus, its minimum
in $\Delta$ is attained in one of its extreme points and Theorem
\ref{thm:occ_meas_realization} implies the result for the unconstrained
case.

For the constrained problem, we observe that set $\Delta_{C}=\left\{ \mu\left(g\right)\in R^{n}|\mu\in\mathcal{M}_{C}\right\} $,
can be expressed as 
\begin{equation}
\Delta_{C}=\Delta\cap\Pi_{1}\cap\cdots\cap\Pi_{m},\label{eq:cnstset_def}
\end{equation}
where for every $k=1,\dots,m$ we define the subspace 
\[
\Pi_{k}=\left\{ \left(z_{1},\dots,z_{n}\right)\in R^{n}|z_{k}\le0\right\} .
\]
The assumption of uniform controllability, implies that $\Delta$
is compact and convex (see Proposition \ref{prop:M_convex_compact}),
and so is $\Delta_{C}$. The function $V\left(\cdot\right)$ is quasi-concave,
so it attains its minimum in an extreme point of $\Delta_{C}$. By
Corollary \ref{Cor:cnv_intesect_halfspace_lemma}, every extreme point
of $\Delta_{C}$ can be expressed as a convex combination of $m+1$
extreme points of $\Delta$, which by Theorem \ref{thm:occ_meas_realization}
correspond to a stationary or periodic solution. This completes the
proof of (a) and (b).

Suppose that $\mu^{*}=\sum_{j=1}^{m+1}\lambda_{j}\mu_{j}^{p}$, is
the minimizing measure, and that each measure $\mu_{j}^{p}$ corresponds
to the periodic or stationary solution $\left(x_{j}^{p}\left(\cdot\right),u_{j}^{p}\left(\cdot\right)\right)$.
The uniform controllability implies that there exists a $T_{K}$ that
bounds the time it takes to steer between any two points in $K$.
We now construct a solution corresponding to the measure $\mu^{*}$
that alternates between the solutions $\left(x_{j}^{p}\left(\cdot\right),u_{j}^{p}\left(\cdot\right)\right)$.

We start at $x_{1}^{p}\left(0\right)$, the initial point of the first
curve, and for $n=1,2,\dots$ we do the following:
\begin{enumerate}
\item for $j=1,\dots,m$ follow the $j$'th curve for $n\lambda_{j}$ units
the time, then steer, in time $\le T_{K}$, to $x_{j+1}^{p}\left(0\right)$,
the initial point of the $\left(j+1\right)$'th curve.
\item Follow the $\left(m+1\right)$'th curve for $n\lambda_{m+1}$ units
of time, then steer, in time $\le T_{K}$, back to $x_{1}^{p}\left(0\right)$,
the initial point of the first curve.
\end{enumerate}
When $m=1$, let us first consider the case when the optimal measure
$\mu^{*}$ is attained by a periodic solution. Let $\left(x^{*}\left(\cdot\right),u^{*}\left(\cdot\right)\right)$
be the corresponding curve and $T^{*}$ its period (if it is stationary
we attribute it a period of $T^{*}=1$). Thus, 
\[
\mu^{*}\left(g^{1}\right)=\frac{1}{T^{*}}\int_{0}^{T^{*}}g^{1}\left(x^{*}\left(t\right),u^{*}\left(t\right)\right)dt\le0.
\]
Setting $\tau$ as the points that maximizes the periodic function
\[
F\left(s\right)=\int_{0}^{s}\left(g^{1}\left(x^{*}\left(t\right),u^{*}\left(t\right)\right)-\mu^{*}\left(g^{1}\right)\right)dt,
\]
and translating, in time, the curve $\left(x^{*}\left(\cdot\right),u^{*}\left(\cdot\right)\right)$
by $\tau$, assures us that $F\left(\cdot\right)$ is non positive,
and that the constraint is satisfied.

Otherwise, the optimal measure is of the form $\mu^{*}=\lambda\mu_{1}+\left(1-\lambda\right)\mu_{2}$,
where each $\mu_{j}\in\mathcal{M}$ corresponds to a periodic trajectory
$\left(x_{j}^{p}\left(\cdot\right),u_{j}^{p}\left(\cdot\right)\right)$,
with period $T_{j}$ (where we attribute stationary solutions with
period 1). In this case, we can assume that $\mu_{1}\left(g^{1}\right)<0<\mu_{2}\left(g^{1}\right)$.

Applying a time translation, we assume that for every $s>0$  the first solution satisfies the integral bound  
$\int_{0}^{s}\left(g^{1}\left(x_{1}^{p}\left(t\right),u_{1}^{p}\left(t\right)\right)-\mu_{1}\left(g^{1}\right)\right)dt\le0$.

In order to accommodate our previous construction to a non averaged
constraint, we take into consideration the transient parts, between
the two curves. To this end we define $\alpha=2M_{g}T_{K}/\left(-\mu_{1}\left(g^{1}\right)\right)$,
where $M_{g}$ bounds $g^{1}\left(x,u\right)$ in $K\times U$.

We start our optimal curve at the point $x_{1}^{p}\left(0\right)$,
then for $n=1,2,\dots$ we repeat the following:
\begin{enumerate}
\item Follow the curve $\left(x_{1}^{p}\left(\cdot\right),u_{1}^{p}\left(\cdot\right)\right)$
for time $T_{1}\left\lceil \frac{\alpha}{T_{1}}\right\rceil +T_{1}\left\lceil \frac{\lambda n}{T_{1}}\right\rceil $
time, where $\left\lceil \beta\right\rceil $ denotes the smallest
integer larger or equal to $\beta$. (The first term takes into account
the transitions between the first and second curves, and is needed
to make sure the constraint holds.).
\item Steers to $x_{2}^{p}\left(0\right)$, in time $\le T_{K}$.
\item Follow the curve $\left(x_{2}^{p}\left(\cdot\right),u_{2}^{p}\left(\cdot\right)\right)$
for $T_{2}\left\lceil \frac{\left(1-\lambda\right)n}{T_{2}}-1\right\rceil $
time.
\item Steers, in time $\le T_{K}$, back to the point $x_{1}^{p}\left(0\right)$.
\end{enumerate}
\end{proof}
\begin{rem}
\label{rem: eq_cnst}The definition of $\sM_{C}$ in Theorem \ref{thm:GEN_THM}
can be relaxed to include equalities as well as inequalities, setting
\[
\sM_{C}=\sM_{1}\cap\cdots\cap\sM_{m},
\]
where for every $k=1,\dots,m$ either $\sM_{k}=\left\{ \mu\in\sM|\mu\left(g_{k}\right)\le0\right\} $
or $\sM_{k}=\left\{ \mu\in\sM|\mu\left(g_{k}\right)=0\right\} $.
\end{rem}
We now present the proofs of Theorems \ref{thm:main theorem} and
\ref{thm:main constraints}.
\begin{proof}
[Proof of Theorem \ref{thm:main theorem}]Let $g\left(x,u\right)\in C\left(K\times U,R^{2}\right)$
be defined by $g^{1}\left(x,u\right)=p\left(x,u\right)$ and $g^{2}\left(x,u\right)=q\left(x,u\right)$.
Set $m=0$ and $V\left(\left(z_{1},z_{2}\right)\right)=z_{1}/z_{2}$.
Since $q\left(x,u\right)>0$, the function $V\left(\cdot\right)$
is quasi-concave in its domain, thus, applying Theorem \ref{thm:GEN_THM}
completes the proof. 
\end{proof}

\begin{proof}
[Proof of Theorem \ref{thm:main constraints}] Let $g\left(x,u\right)\in C\left(K\times U,R^{2+m}\right)$
be defined by $g^{k}\left(x,u\right)=C_{k}\left(x,u\right)$ for $k=1,\dots,m$
and $g^{m+1}\left(x,u\right)=p\left(x,u\right)$, $g^{m+2}\left(x,u\right)=q\left(x,u\right)$.
Set $n=m+2$ and $V\left(\left(z_{1},z_{2},\dots,z_{m+2}\right)\right)=z_{m+1}/z_{m+2}$.
The function $V\left(\cdot\right)$ is quasi-concave in its domain,
thus, combining Theorem \ref{thm:GEN_THM} and Remark \ref{rem: eq_cnst}
completes the proof. 
\end{proof}
The necessity to alternate between stationary solutions of the constrained
optimization problem is depicted in the following scalar examples.
\begin{example}
\label{ex:1}Suppose the control system (\ref{eq:ODE}) is defined
by $f\left(x,u\right)=u$ and $K=U=\left[1,-1\right]$. The minimization
of 
\[
v^{*}=\limsup_{T\rightarrow\infty}\frac{1}{T}\int_{0}^{T}\left(1-x\left(t\right)-\left(x\left(t\right)\right)^{2}\right)dt,
\]
with respect to feasible solutions satisfying $\int_{0}^{T}x\left(t\right)dt\le0$
for every $T>0$, attains the optimal value $v^{*}=0$ by a solution
that alternates between $-1$ and $1$. This follows from the fact
that substituting the constraint in the value, we see that for every
$T$ 
\begin{equation}
\int_{0}^{T}\left(1-x\left(t\right)-\left(x\left(t\right)\right)^{2}\right)dt=\int_{0}^{T}\left(1-\left(x\left(t\right)\right)^{2}\right)dt-\int_{0}^{T}x\left(t\right)dt\ge T-\int_{0}^{T}\left(x\left(t\right)\right)^{2}dt\ge0,\label{eq:value larger than 0-1}
\end{equation}
which implies $v^{*}\ge0$. Moreover, one concludes from (\ref{eq:value larger than 0-1})
that there is no periodic optimal solution.
\end{example}

\begin{example}
When the value function $V\left(\cdot\right)$ in Theorem \ref{thm:GEN_THM}
is not concave, the necessity to alternate between stationary solutions
is depicted in the following example. The minimization of
\[
v^{*}=\limsup_{T\rightarrow\infty}\frac{1}{T}\left(\left|\int_{0}^{T}x\left(t\right)dt\right|-\int_{0}^{T}\left|x\left(t\right)\right|dt\right),
\]
with respect to the system in Example \ref{ex:1}, is not attained
by a periodic solution. Clearly $v^{*}\ge-1$, and the minimum is
attained by a solution alternating between the points $+1$ and $-1$.
Notice that the function $V\left(\cdot\right)$ is convex and the
maximization problem is attained by the stationary solution $x\equiv0$.
(as the triangle inequality implies that $v^{*}\le0$ for every feasible
solution).
\end{example}

The author wishes to thank Sergey Rossomakhine for conducting the numerical computations of the Cheeger set and Cheeger constant. 

\bibliographystyle{plain}
\bibliography{citations}

\end{document}